\documentclass[11pt]{amsart}
\usepackage[T1]{fontenc}
\usepackage{amsmath,amssymb,mathtools}
\usepackage{newtxtext,newtxmath}
\usepackage{microtype,booktabs}
\usepackage[a4paper,margin=29mm,headheight=14pt]{geometry}
\usepackage[colorlinks=true,linkcolor=blue!45!black,citecolor=blue!45!black,urlcolor=blue!45!black]{hyperref}
\usepackage{xcolor}
\numberwithin{equation}{section}
\newtheorem{theorem}{Theorem}[section]
\newtheorem{corollary}[theorem]{Corollary}
\newtheorem{proposition}[theorem]{Proposition}
\newtheorem{lemma}[theorem]{Lemma}
\theoremstyle{remark}

\allowdisplaybreaks[2]
\hypersetup{pdftitle={Quantitative hyperconvexity and the Bergman distance},pdfauthor={Bo-Yong Chen and Yuanpu Xiong},pdfsubject={Arbitrary boundary moduli and quantitative Bergman completeness}}

\title{Quantitative hyperconvexity and the Bergman distance}
\date{}
\author{Bo-Yong Chen and Yuanpu Xiong}

\begin{document}

\address [Bo-Yong Chen] {School of Mathematical Sciences, Fudan University, Shanghai, 200433, China}
\email{boychen@fudan.edu.cn}

\address [Yuanpu Xiong] {School of Mathematical Sciences and Key Laboratory of Intelligent Computing and Applications (Tongji University), Ministry of Education, Tongji University, Shanghai, 200092, CHINA}
\email{ypxiong@tongji.edu.cn}

\begin{abstract}
Let $\Omega\subset\mathbb{C}^n$ be a bounded hyperconvex domain admitting a negative plurisubharmonic exhaustion function $u$ such that $-u(z)\le C\exp[-\eta(\log(1/\delta_\Omega(z)))]$ near the boundary, where $\eta$ is nondecreasing with $\lim_{t\rightarrow+\infty}\eta(t)=+\infty$. We present a lower bound for the Bergman distance $d_{B,\Omega}(z_0,z)$ by counting the iterates of a function determined by the generalized inverse of $\eta$. The bound tends to infinity as $z\rightarrow\partial\Omega$ for every such $\eta$; moreover, it recovers several known estimates. The proof combines a one-sided localization of Green sublevel sets with boundary estimates for the Bergman kernel and a capacity--volume inequality. We also obtain an upper bound for the Bergman distance along the axis of a family of complete Hartogs domains whose fiber radii decay faster than every positive power of the boundary distance in the base disc.
\end{abstract}
\maketitle

\section{Introduction}

For a domain $\Omega\subset\mathbb{C}^n$, let $K_\Omega(z,w)$ denote the Bergman kernel and set $K_\Omega(z)=K_\Omega(z,z)$. One may define
\begin{equation}\label{eq:metric}
\beta_\Omega(z;X)^2
=\sum_{j,k=1}^n\frac{\partial^2\log K_\Omega(z)}{\partial z_j\partial\bar z_k}X_j\overline{X_k},\qquad z\in\Omega,\ X\in\mathbb{C}^n,
\end{equation}
when $K_\Omega(z)>0$ for all $z\in\Omega$. $\beta_\Omega(z;X)^2$ is called the Bergman metric when it is positive definite. The integrated distance is the Bergman distance $d_{B,\Omega}$. We say the domain is Bergman complete if $(\Omega,d_{B,\Omega})$ is a complete metric space.

It is well-known that every hyperconvex domain, that is, it admits a continuous negative plurisubharmonic exhaustion function, is Bergman complete (cf. \cite{BP98},  \cite{Herbort99}; see also  \cite{C04}). To obtain a quantitative estimate for the Bergman distance on bounded hyperconvex domains, it is natural to impose boundary decay conditions on exhaustion functions.

Diederich-Ohsawa \cite{DO95} first obtained the effective estimate
\begin{equation}\label{eq:distance_estimate_1}
d_{B,\Omega}(z,z_0)\geq c\log\log\frac{1}{\delta_\Omega(z)},\ \ \ z\rightarrow\partial\Omega,
\end{equation}
under the condition that there exists a continuous negative plurisubharmonic exhaustion function $u$ with
\begin{equation}\label{eq:DO-condition}
C^{-1}\delta_\Omega(z)^\alpha
\le -u
\le C \delta_\Omega(z)^\beta
\end{equation}
for some constants $C>0$ and $\alpha>\beta>0$. Here $\delta_\Omega(z)$ denotes the Euclidean distance from $z$ to $\partial\Omega$. This estimate was improved by B\l ocki \cite{B05} to
\begin{equation}\label{eq:power-background}
d_{B,\Omega}(z_0,z)\ge c\frac{\log(1/\delta_\Omega(z))}{\log \log(1/\delta_\Omega(z))}
\end{equation}
in the case $\alpha=\beta$ in \eqref{eq:DO-condition}, which is satisfied, e.g., when $\Omega$ is a bounded pseudoconvex domain with $C^2$ boundary (cf.  \cite{DF77}). The first author of the present paper proved that \eqref{eq:power-background} still holds under the weaker condition that the hyperconvexity index of $\Omega$ is positive, i.e.,
\begin{equation}\label{eq:Chen-condition}
-u\le C\delta_\Omega^\alpha,
\end{equation}
with $\alpha>0$ (cf. \cite{C17}). These results suggest that different rates of decay of $-u$ should imply different lower bounds for the Bergman distance. A partial result was obtained in \cite{CZ22}, which states that
\begin{equation}\label{eq:triple-log}
d_{B,\Omega}(z_0,z)\ge c\log\log\log\frac{1}{\delta_\Omega(z)}
\end{equation}
holds when $z$ is sufficiently close to $\partial\Omega$, if there exists a continuous negative plurisubharmonic exhaustion function $u$ such that
\begin{equation}\label{eq:CZ-condition}
-u\le C (\log(1/\delta_\Omega(z)))^{-\alpha},\quad z\to \partial \Omega,
\end{equation}
where $\alpha> (n-1+\sqrt{(n-1)(n+3)})/2$. It is natural to expect that this condition may be replaced by $\alpha>0$.

There are three major ingredients in the estimates of the Bergman distance: the first is Kobayashi's characterization of the Bergman metric through embedding $\Omega$ into the infinite dimensional projective Hilbert space via a/any complete orthonormal basis of $A^2(\Omega)$ (cf. \cite{Kobayashi59}); the second is weighted $L^2$ estimates for $\bar{\partial}$;  the third is the estimate of the sublevel sets of the pluricomplex Green function, which relies heavily on the ``quasi-symmetry'' of the pluricomplex Green function (see  \cite[Theorem 5.1]{B05}). The limitations of these ``quasi-symmetry'' estimates make it difficult to treat exhaustion functions with slower boundary decay.

The purpose of this paper is to treat an arbitrary nondecreasing decay function. The estimate is most naturally expressed by a recurrence.

\begin{theorem}\label{thm:main}
Let $\Omega\subset\mathbb{C}^n$ be a bounded hyperconvex domain and fix $z_0\in\Omega$. Suppose that $u<0$ is a plurisubharmonic exhaustion function and that, near $\partial\Omega$,
\begin{equation}\label{eq:assumption}
-u(z)\le C_u\exp\left[-\eta\left(\log\frac{1}{\delta_\Omega(z)}\right)\right],
\end{equation}
where $\eta:[t_0,+\infty)\rightarrow(0,+\infty)$ is nondecreasing,
$t_0\ge1$, and $\eta(t)\rightarrow+\infty$ as $t\rightarrow+\infty$. For $s\in(0,+\infty)$ define the generalized inverse function of $\eta$ by
\begin{equation}\label{eq:inverse}
\vartheta(s)=\sup\bigl(\{t_0\}\cup\{t\ge t_0:\eta(t)\le s\}\bigr).
\end{equation}
There are constants $b>0$ and $\tau_0>0$ such that, for the sequence $\{\tau_k\}$ given by
\begin{equation}\label{eq:recurrence}
\tau_{k+1}=\tau_k+b+\log(1+\vartheta(\tau_k)),\qquad k=0,1,2,\cdots,
\end{equation}
if $z$ is sufficiently close to $\partial\Omega$ and $\eta(\log(1/\delta_\Omega(z)))\geq \tau_k$, then
\begin{equation}\label{eq:main}
d_{B,\Omega}(z_0,z)\ge ck-C.
\end{equation}
Here $c,C>0$ are independent of $z$ and $k$.
\end{theorem}

Notice that $\tau_k\rightarrow+\infty$ as $k\rightarrow+\infty$. We can choose arbitrarily large $k=k(z)$ for $z$ sufficiently close to $\partial\Omega$. Consequently,
\[
\lim_{z\rightarrow\partial\Omega}d_{B,\Omega}(z_0,z)=+\infty,
\]
which also recovers the Bergman completeness of a bounded hyperconvex domain. Indeed, if $u$ is a negative plurisubharmonic exhaustion function on $\Omega$, then we may take
\[
\eta(t)=-\log\sup_{\delta_\Omega\le e^{-t}}(-u)
\]
for sufficiently large $t$.

An integral estimate follows from Theorem~\ref{thm:main}.

\begin{corollary}\label{cor:integral}
Under the assumptions of Theorem~\ref{thm:main}, there is $s_0>0$
such that for $z$ sufficiently close to $\partial\Omega$
\begin{equation}\label{eq:integral}
d_{B,\Omega}(z_0,z)
\gtrsim\int_{s_0}^{\eta(\log(1/\delta_\Omega(z)))}\frac{\mathrm ds}{1+\log \vartheta(s)}.
\end{equation}
If, in addition, $\eta$ is continuously differentiable and strictly increasing, then, for a suitable $t_0$,
\begin{equation}\label{eq:differential}
d_{B,\Omega}(z_0,z)
\gtrsim\int_{t_0}^{\log(1/\delta_\Omega(z))}\frac{\eta'(t)}{1+\log t}\,\mathrm{d} t
\end{equation}
holds for $z$ sufficiently close to $\partial\Omega$.
\end{corollary}

Here and below $f\gtrsim g$ (resp. $f\lesssim g$) means that $f\ge cg$ (resp. $f\le C g$) with constants independent of $z$. In particular, \eqref{eq:differential} gives
\begin{align}\label{eq:eta-loglog}
d_{B,\Omega}(z_0,z)
&\gtrsim\frac{1}{1+\log\log(1/\delta_\Omega(z))}\int_{t_0}^{\log(1/\delta_\Omega(z))}\eta'(t)\,\mathrm{d} t \notag\\
&\gtrsim\frac{\eta(\log(1/\delta_\Omega(z)))}{\log\log(1/\delta_\Omega(z))}
\end{align}
under the additional assumption that
\[
\lim_{t\rightarrow+\infty}\frac{\eta(t)}{\log{t}}=+\infty.
\]
When $\eta(t)=\alpha t$, i.e., $\Omega$ has positive hyperconvexity index, \eqref{eq:eta-loglog} becomes the estimate \eqref{eq:power-background} in \cite{C17}. \eqref{eq:eta-loglog} is weaker than \eqref{eq:differential}. For example, when $\eta(t)=\alpha\log t$,
\begin{align}\label{eq:triple}
d_{B,\Omega}(z_0,z)
&\gtrsim \int_{t_0}^{\log(1/\delta_\Omega(z))}\frac{1}{t(1+\log t)}\,\mathrm{d} t \notag\\
&\gtrsim\log\log\log\frac1{\delta_\Omega(z)}.
\end{align}

The main point of the proof of Theorem \ref{thm:main} is to estimate the mass of a normalized Bergman kernel on a Green sublevel set. We first obtain a one-sided localization in terms of a relative extremal function, and then divide the sublevel set into a boundary part, where an $L^2$ estimate from \cite{C17} applies, and an interior part, whose volume is controlled by its Monge--Amp\`ere capacity (cf. \cite{ACKPZ09}). This gives a fixed positive Bergman distance after at most $n+1$ iterations of $F$. Counting successive levels along an arbitrary curve proves the theorem.

There are many domains satisfying \eqref{eq:assumption} for various choices of $\eta$. Namely, we have

\begin{theorem}\label{thm:eta-choice}
For every nondecreasing function $\eta:[t_0,\infty)\to(0,\infty)$  with $\eta(t)\rightarrow+\infty$ and $\eta(t)=o(t)$ as $t\rightarrow+\infty$,   there exists a Hartogs domain $\Omega\subset\mathbb{C}^2$ with a continuous negative plurisubharmonic exhaustion $u$ satisfying \eqref{eq:assumption} near $\partial\Omega$. Moreover, the hyperconvexity index of $\Omega$ is zero.
\end{theorem}

The lower bound in Theorem 1.1 is probably optimal. To explore this possibility, we consider the following family of Hartogs domains.

\begin{theorem}\label{thm:hartogs-alpha}
For $\alpha>1$, put
\begin{equation}\label{eq:hartogs-alpha-domain}
\Omega_\alpha=\left\{(z,w)\in\mathbb C^2:|z|<1,\quad |w|<\exp\left[-\left(\log\frac{e}{1-|z|^2}\right)^\alpha\right]\right\}.
\end{equation}
Let $\alpha>1$ and $P_t=(1-t,0)$, where $0<t<1/2$. Then, as
$t\to0^+$,
\begin{equation}\label{eq:hartogs-alpha-distance-delta}
\frac{(\log(1/\delta_{\Omega_\alpha}(P_t)))^{1/\alpha}}{\log\log(1/\delta_{\Omega_\alpha}(P_t))}
\lesssim d_{B,\Omega_\alpha}((0,0),P_t)
\lesssim\left(\log\frac1{\delta_{\Omega_\alpha}(P_t)}\right)^{(\alpha+1)/(2\alpha)}.
\end{equation}
\end{theorem}

It remains an interesting question to find the optimal estimates.

\section{A one-sided estimate for Green sublevel sets}\label{sec:Green}

The main result in this section, Proposition \ref{prop:localization}, is a modification of \cite[Proposition 2.5]{C17} and \cite[Lemma 2.2]{CZ22}. The proof is also analogous. We include the details here for the reader's convenience.

Fix a closed Euclidean ball $B$ centered at $z_0$, with
$B\subset\Omega$, and let
\[
\rho(z):=\sup\{v(z):v\in\mathrm{PSH}(\Omega),\ v\le0,\ v|_B\le-1\}
\]
be its relative extremal function. For a closed ball in a bounded hyperconvex domain, this envelope agrees with its upper semicontinuous regularization and is continuous on
$\overline\Omega$ after extension by zero to the boundary; see \cite[Section~1]{C17} and \cite[Proposition~3.1.3(vii)]{B02}. In particular,
\begin{equation}\label{eq:rho-properties}
\rho\in\mathrm{PSH}(\Omega)\cap C(\overline\Omega),\qquad -1\le\rho<0,\qquad \rho|_B=-1,\qquad \rho|_{\partial\Omega}=0.
\end{equation}

We first show that the decay estimate \eqref{eq:assumption} also holds for $\rho$. The function $-u$ is positive and lower semicontinuous. Consequently $m_B:=\inf_B(-u)>0$, and $u/m_B\le-1$ on $B$. It follows that
\[
\rho\ge u/m_B.
\]
By \eqref{eq:assumption}, after choosing $C_B:=\max\{1,C_u/m_B\}$ sufficiently large, we obtain
\begin{equation}\label{eq:rho-decay}
-\rho(z)\le C_Be^{-\eta(\log(1/\delta_\Omega(z)))}
\end{equation}
near $\partial\Omega$, so that
\begin{equation}\label{eq:rho-decay2}
\qquad
\log\left(-\frac1{\rho(z)}\right)\ge\eta\left(\log\frac{1}{\delta_\Omega(z)}\right)-A,\qquad A=\log C_B.
\end{equation}
This comparison does not require $u$ to be continuous. Notice also that $\log(-1/\rho)$ is continuous and nonnegative on $\Omega$, and vanishes on $B$.

\begin{lemma}\label{lem:translation}
For every fixed $r>1$, there is $\varepsilon_r>0$ such that
\begin{equation}\label{eq:translation}
\rho(z_2)
\ge r\rho(z_1)-C_B\exp\left[-\eta\left(\log\frac1{|z_1-z_2|}\right)\right]
\end{equation}
whenever $z_1,z_2\in\Omega$ and
$0<|z_1-z_2|<\varepsilon_r$.
\end{lemma}

\begin{proof}
Choose $0<\varepsilon_r<e^{-t_0}$ so small that \eqref{eq:rho-decay} holds when $\delta_\Omega(z)<\varepsilon_r$. Since $\rho=-1$ on the compact ball $B$ and $\rho$ is continuous, we may decrease $\varepsilon_r$ so that the closed $\varepsilon_r$-neighborhood of $B$ lies in $\Omega$ and
\begin{equation}\label{eq:ball-neighborhood}
\rho(\zeta)\le-1/r
\end{equation}
for all $\zeta$ with $\mathrm{dist}(\zeta,B)\le\varepsilon_r$. For $0<|\tau|<\varepsilon_r$, put
\[
\Omega_\tau:=\{\zeta\in\Omega:\zeta+\tau\in\Omega\},\qquad E_\tau:=C_B\exp\left[-\eta\left(\log\frac1{|\tau|}\right)\right],
\]
and define
\begin{equation}\label{eq:translated-competitor}
v_\tau(\zeta)=\begin{cases}
\max\{\rho(\zeta),r\rho(\zeta+\tau)-E_\tau\},\quad &\zeta\in\Omega_\tau,\\
\rho(\zeta),\quad &\zeta\in\Omega\setminus\Omega_\tau.
\end{cases}
\end{equation}
Clearly $v_\tau<0$. We first show that $v_\tau$ is a plurisubharmonic function on $\Omega$. If $\xi\in\partial\Omega_\tau\cap\Omega$, then $\xi+\tau\in\partial\Omega$, and hence $\delta_\Omega(\xi)\le|\tau|<\varepsilon_r$.
By \eqref{eq:rho-decay} and monotonicity of $\eta$,
\[
-\rho(\xi)
\le C_B\exp\left[-\eta\left(\log\frac1{\delta_\Omega(\xi)}\right)\right]
\le E_\tau.
\]
Moreover, $\rho(\zeta+\tau)\to0$ as $\zeta\to\xi$ within $\Omega_\tau$. Therefore
\[
\limsup_{\substack{\zeta\to\xi\\\zeta\in\Omega_\tau}}\bigl(r\rho(\zeta+\tau)-E_\tau\bigr)
=-E_\tau\le\rho(\xi),
\]
so that $v_\tau\in\mathrm{PSH}(\Omega)$.

For $\zeta\in B$, the translated point $\zeta+\tau$ lies in the $\varepsilon_r$-neighborhood of $B$. Hence
\[
\rho(\zeta)=-1, \qquad r\rho(\zeta+\tau)-E_\tau\le-1-E_\tau\le-1.
\]
Thus $v_\tau=-1$ on $B$, and the extremal property gives $v_\tau\le\rho$ throughout $\Omega$. Finally, take $\tau=z_1-z_2$. Then $z_2\in\Omega_\tau$ and $z_2+\tau=z_1$, whence
\[
\rho(z_2)\ge v_\tau(z_2)\ge r\rho(z_1)-E_\tau,
\]
which proves \eqref{eq:translation}.
\end{proof}

Let $g_\Omega(\cdot,z)$ be the negative pluricomplex Green function with pole at $z$, defined by
\[
g_\Omega(\zeta,z)=\sup\{v(\zeta):v\in\mathrm{PSH}(\Omega),\ v\le0,\ v(\xi)\le\log|\xi-z|+O(1)\text{ near }z\}.
\]
Set $A_\Omega(z,-1)=\{\zeta\in\Omega:g_\Omega(\zeta,z)<-1\}$.

\begin{proposition}\label{prop:localization}
There exist $a>0$ and $c_0\in(0,1)$ such that
\begin{equation}\label{eq:one-sided}
A_\Omega(z,-1)\subset\{\zeta\in\Omega:-\rho(\zeta)>\mu(z)\},
\end{equation}
where
\begin{equation}\label{eq:mu}
\mu(z):=\frac{-c_0\rho(z)}{1+\vartheta(\log(-1/\rho(z))+a)}, \quad z\in\Omega.
\end{equation}
\end{proposition}

\begin{proof}
Choose $0<r_0<e^{-t_0}$ small enough for Lemma~\ref{lem:translation} to apply with $r=3/2$, and so that \eqref{eq:rho-decay} holds when $\delta_\Omega(z)<r_0$. Fix
\begin{equation}\label{eq:a-choice}
a\ge\max\{\log(4C_B),\eta(\log(1/r_0)),1\}.
\end{equation}
We first show that
\begin{equation}\label{eq:interface-distance}
\rho(\zeta)=\frac12\rho(z)\quad\Longrightarrow\quad
\log\frac1{|\zeta-z|}
\le \vartheta\left(\log\left(-\frac{1}{\rho(z)}\right)+a\right).
\end{equation}
If $|\zeta-z|<r_0$, use \eqref{eq:translation} with $z_1=\zeta$, $z_2=z$ and $r=3/2$. Since $\frac32\rho(\zeta)=\frac34\rho(z)$, we get
\[
-\frac{\rho(z)}{4}
\le C_B\exp\left[-\eta\left(\log\frac1{|\zeta-z|}\right)\right].
\]
Taking logarithms yields
\[
\eta\left(\log\frac1{|\zeta-z|}\right)
\le\log\left(-\frac{1}{\rho(z)}\right)+\log(4C_B)
\le\log\left(-\frac{1}{\rho(z)}\right)+a,
\]
which proves \eqref{eq:interface-distance} in this case as $\log(1/|\zeta-z|)>t_0$. If $|\zeta-z|\ge r_0$, then
\[
\log\frac1{|\zeta-z|}\le\log\frac1{r_0}
\le \vartheta\left(\log\left(-\frac{1}{\rho(z)}\right)+a\right),
\]
because $\eta(\log(1/r_0))\le a$ and $|\rho(z)|\le1$. Thus \eqref{eq:interface-distance} holds for every $z\in\Omega$.

Fix $R>1+\mathrm{diam}(\Omega)$, and for the fixed pole $z$ write
\[
L_z:=1+\log R+\vartheta\left(\log\left(-\frac{1}{\rho(z)}\right)+a\right).
\]
Define
\begin{equation}\label{eq:green-competitor}
\psi_z(\zeta)=\begin{cases}
\log\dfrac{|\zeta-z|}{R},\quad &\rho(\zeta)\le\rho(z)/2,\\[2pt]
\max\left\{\log\dfrac{|\zeta-z|}{R},\dfrac{2L_z\rho(\zeta)}{-\rho(z)}\right\},\quad &\rho(\zeta)>\rho(z)/2.
\end{cases}
\end{equation}
On the interface $\rho(\zeta)=\rho(z)/2$, \eqref{eq:interface-distance} gives
\[
\log\frac{|\zeta-z|}{R}\ge1-L_z,\qquad \frac{2L_z\rho(\zeta)}{|\rho(z)|}=-L_z.
\]
It follows that $\psi_z\in\mathrm{PSH}(\Omega)$. Moreover, since $\psi_z<0$ and $\psi_z(\zeta)=\log|\zeta-z|-\log R$ when $\zeta$ lies in a neighborhood of $z$, we infer from the definition of the pluricomplex Green function that
\begin{equation}\label{eq:green-comparison}
g_\Omega(\zeta,z)\ge\psi_z(\zeta),\qquad \zeta\in\Omega.
\end{equation}

Now let $\zeta\in A_\Omega(z,-1)$. If $\rho(\zeta)\le\rho(z)/2$, then $-\rho(\zeta)\ge-\rho(z)/2>-\rho(z)/(2L_z)$, because $L_z>1$. In the other case, \eqref{eq:green-comparison} gives
\[
-1
>g_\Omega(\zeta,z)
\ge\psi_z(\zeta)
\ge\frac{2L_z\rho(\zeta)}{-\rho(z)},
\]
which again implies $-\rho(\zeta)>-\rho(z)/(2L_z)$. Finally,
\[
L_z\le(1+\log R)\bigl(1+\vartheta(\log(1/|\rho(z)|)+a)\bigr).
\]
Taking $c_0=[2(1+\log R)]^{-1}$, we obtain
\[
-\rho(\zeta)
>\frac{-\rho(z)}{2L_z}
\ge\mu(z),
\]
as asserted.
\end{proof}

Only the one-sided inclusion \eqref{eq:one-sided} will be used. In particular, the argument does not require symmetry of the pluricomplex Green function or disjointness of sublevel sets with different poles.

\section{Boundary mass and interior capacity}\label{sec:MA}
For $z\in\Omega$, put
\[
k_z(\zeta)=\frac{K_\Omega(\zeta,z)}{\sqrt{K_\Omega(z)}}.
\]
The reproducing property gives $\|k_z\|_{L^2(\Omega)}=1$. We shall estimate its mass on $A_\Omega(w,-1)$ by separating a boundary layer from the region where $-\rho$ is bounded below. First, write
\[
\Omega_a:=\{\zeta\in\Omega:-\rho(\zeta)>a\}.
\]
The following lemma is proved in \cite[Proposition 2.1]{C17} (see also \cite[Proposition 6.1]{C16}).

\begin{lemma}\label{lem:boundary-mass-0}
For every $0<\beta<1$, there are constants $C_\beta,\varepsilon_\beta>0$ such that
\begin{equation}\label{eq:boundary-mass-0}
\int_{\{-\rho\le\varepsilon\}}|K_\Omega(\cdot,z)|^2\,\mathrm{d}V
\le C_\beta K_{\Omega_a}(z)\left(\frac{\varepsilon}{a}\right)^\beta,\qquad 0<\varepsilon\le\varepsilon_\beta a,
\end{equation}
for all $z\in\Omega_a$.
\end{lemma}

Take $a=\mu(z)$, so that Proposition \ref{prop:localization} gives $A_\Omega(z,-1)\subset\Omega_a$. Then
\[
\int_{\{-\rho\le\varepsilon\}}|k_z|^2\,\mathrm{d}V
\le C_\beta \frac{K_{A_\Omega(z,-1)}(z)}{K_\Omega(z)}\left(\frac{\varepsilon}{\mu(z)}\right)^\beta,\qquad 0<\varepsilon\le\varepsilon_\beta\mu(z).
\]
Next, we will make use of the following lemma (cf. \cite{C99,Herbort99}).

\begin{lemma}\label{lem:localization1}
There exists a constant $L_n>0$ such that
\[
K_\Omega(z)\leq K_{A_\Omega(z,-1)}(z)\leq L_n K_\Omega(z),\qquad \forall\,z\in\Omega.
\]
\end{lemma}

Consequently,
\begin{equation}\label{eq:boundary-mass}
\int_{\{-\rho\le\varepsilon\}}|k_z|^2\,\mathrm{d}V
\le C_*\left(\frac{\varepsilon}{\mu(z)}\right)^\beta,\qquad 0<\varepsilon\le\varepsilon_\beta\mu(z),
\end{equation}
for all $z\in\Omega$, where $C_*=C_\beta L_n$.

In what follows, we normalize $dd^c=(i/\pi)\partial\bar\partial$, so that
\[
\bigl(dd^cg_\Omega(\cdot,w)\bigr)^n=\delta_w,
\]
where $\delta_w$ denotes the Dirac measure. For a Borel set $E\subset\Omega$, let
\[
\mathrm{Cap}_\Omega(E)
=\sup\left\{\int_E(dd^c\psi)^n:\psi\in\mathrm{PSH}(\Omega),\ -1\le\psi\le0\right\}.
\]
Here, all Monge--Amp\`ere measures of locally bounded functions are understood in the Bedford--Taylor sense.

\begin{lemma}\label{lem:capacity}
For $w\in\Omega$ and $\varepsilon>0$,
\begin{equation}\label{eq:capacity}
\mathrm{Cap}_\Omega\bigl(A_\Omega(w,-1)\cap\{-\rho>\varepsilon\}\bigr)
\le\frac{n!}{\varepsilon}\bigl(-\rho(w)\bigr).
\end{equation}
This estimate does not use \eqref{eq:assumption}.
\end{lemma}

\begin{proof}
Fix a test function $\psi\in\mathrm{PSH}(\Omega)$ with
$-1\le\psi\le0$, and set
\[
v=\max\{\rho/\varepsilon,\psi\},\qquad
g_j=\max\{g_\Omega(\cdot,w),-j\},\quad j=1,2,\ldots.
\]
Then $-1\le v\le0$, $-v\le -\rho/\varepsilon$, and $v=\psi$ on the open set $\{-\rho>\varepsilon\}$. The functions $g_j$ are bounded, negative and plurisubharmonic, and have zero boundary values. B\l ocki's integral inequality \cite[Theorem~2.1 and Corollary~2.2]{B93} therefore gives
\begin{equation}\label{eq:Blocki}
\int_\Omega(-g_j)^n(dd^cv)^n
\le n!\int_\Omega(-v)(dd^cg_j)^n
\le\frac{n!}{\varepsilon}\int_\Omega (-\rho)(dd^cg_j)^n.
\end{equation}
The standard truncation property of a Green function says that $(dd^cg_j)^n$ has total mass one and is supported on $\{g_\Omega(\cdot,w)=-j\}$; see \cite[Section~3.3]{B02}. It follows from the well-known estimate
\[
g_\Omega(\zeta,w)\ge\log\frac{|\zeta-w|}{R},
\qquad \text{where}\ R>1+\mathrm{diam}(\Omega),
\]
that
\[
\{g_\Omega(\cdot,w)\le-j\}\subset\overline{B(w,Re^{-j})}.
\]
Since these balls are shrinking to $\{w\}$ as $j\rightarrow+\infty$, it follows that
\begin{equation}\label{eq:Blocki_right}
\lim_{j\rightarrow+\infty}\int_\Omega (-\rho)(dd^cg_j)^n=-\rho(w).
\end{equation}
On the other hand, $(-g_j)^n$ increases pointwise to $(-g_\Omega(\cdot,w))^n$. \eqref{eq:Blocki}, \eqref{eq:Blocki_right} together with the monotone convergence theorem yield
\[
\int_\Omega(-g_\Omega(\cdot,w))^n(dd^cv)^n
\le\frac{n!}{\varepsilon}\bigl(-\rho(w)\bigr).
\]
On $A_\Omega(w,-1)$, the factor $(-g_\Omega)^n$ is larger than one, and on the open subset $\{-\rho>\varepsilon\}$ we have $(dd^cv)^n=(dd^c\psi)^n$. Hence
\[
\int_{A_\Omega(w,-1)\cap\{|\rho|>\varepsilon\}}(dd^c\psi)^n
\le\frac{n!}{\varepsilon}\bigl(-\rho(w)\bigr).
\]
Taking the supremum over $\psi$ proves \eqref{eq:capacity}.
\end{proof}

\begin{lemma}\label{lem:volume}
There are constants $C_V,c_V>0$ such that
\begin{equation}\label{eq:volume}
\mathrm{Vol}\bigl(A_\Omega(w,-1)\cap\{-\rho>\varepsilon\}\bigr)
\le C_V\exp\left[-c_V\left(\frac{\varepsilon}{-\rho(w)}\right)^{1/n}\right]
\end{equation}
for every $w\in\Omega$ and $\varepsilon>0$.
\end{lemma}

\begin{proof}
When $\mathrm{Cap}_\Omega(A_\Omega(w,-1)\cap\{-\rho>\varepsilon\})=0$, we have $\mathrm{Vol}(A_\Omega(w,-1)\cap\{-\rho>\varepsilon\})=0$, so that \eqref{eq:volume} holds trivially. On the other hand, the capacity--volume estimate of \cite[Proposition~6.1]{ACKPZ09} gives, for a Borel set $E$ of finite positive capacity,
\[
\mathrm{Vol}(E)
\le c_n(\mathrm{diam}(\Omega))^{2n}\bigl(1+\mathrm{Cap}_\Omega(E)^{-1/n}\bigr)^{n-1}\exp\bigl(-2n\mathrm{Cap}_\Omega(E)^{-1/n}\bigr).
\]
Since $(1+x)^{n-1}e^{-2nx}\le C_ne^{-nx}$ for $x\ge0$, \eqref{eq:capacity} implies \eqref{eq:volume} with
\[
C_V=c_nC_n(\mathrm{diam}(\Omega))^{2n},\qquad c_V=\frac{n}{(n!)^{1/n}}. \qedhere
\]
\end{proof}

By \eqref{eq:boundary-mass} and Lemma \ref{lem:volume}, we have
\begin{align}\label{eq:split-mass-0}
\int_{A_\Omega(w,-1)}|k_z|^2\,\mathrm{d} V
&\le C_*\left(\frac{\varepsilon}{\mu(z)}\right)^\beta\\
&\quad+C_V\exp\left[-c_V\left(\frac{\varepsilon}{-\rho(w)}\right)^{1/n}\right]\sup_{\{-\rho>\varepsilon\}}|k_z|^2. \notag
\end{align}
Notice that
\begin{equation}\label{eq:ball-kernel}
|k_z(\zeta)|^2
\le K_\Omega(\zeta)
\le\frac{n!}{\pi^n\delta_\Omega(\zeta)^{2n}}.
\end{equation}
Recall that \eqref{eq:rho-decay} holds when $\delta_\Omega<r_0$. If $-\rho(\zeta)>\varepsilon$ and $\delta_\Omega(\zeta)<r_0$, then \eqref{eq:rho-decay} yields
\[
\eta\left(\log\frac{1}{\delta_\Omega(\zeta)}\right)
< \log\frac{1}{\varepsilon}+\log C_B
\le\log\frac{1}{\varepsilon}+a,
\]
i.e.,
\[
\frac{1}{\delta_\Omega(\zeta)}
\le \exp\left[\vartheta\left(\log\frac{1}{\varepsilon}+a\right)\right].
\]
If $\delta_\Omega(\zeta)\ge r_0$, by \eqref{eq:a-choice} we have
\[
\frac{1}{\delta_\Omega(\zeta)}\le \frac{1}{r_0}
\le\exp\left[\vartheta(a)\right].
\]
These combined with \eqref{eq:ball-kernel} imply
\begin{equation}\label{eq:interior-kernel}
\sup_{\{-\rho>\varepsilon\}}|k_z|^2
\le \frac{n!}{\pi^n}\exp\left[2n\vartheta\left(\log\frac{1}{\varepsilon}+a\right)\right],\qquad z\in\Omega.
\end{equation}
By \eqref{eq:split-mass-0} and \eqref{eq:interior-kernel}, we obtain the estimate used below:
\begin{align}\label{eq:split-mass}
\int_{A_\Omega(w,-1)}|k_z|^2\,\mathrm{d} V
&\le C_*\left(\frac{\varepsilon}{\mu(z)}\right)^\beta\\
&\quad+C_{**}\exp\left[2n\vartheta\left(\log\frac{1}{\varepsilon}+a\right)-c_V\left(\frac{\varepsilon}{-\rho(w)}\right)^{1/n}\right], \notag
\end{align}
where $C_{**}=\tfrac{n!}{\pi^n}C_V$ and $0<\varepsilon<1$, $\varepsilon\le\varepsilon_\beta\mu(z)$.

\section{Separation in the Bergman distance}\label{sec:Bergman}
\begin{proposition}\label{prop:separation}
With $a$ chosen as above, we can choose sufficiently large $b$ such that for
\[
F(s)=s+b+\log(1+\vartheta(s+a))
\]
and any $z,w\in\Omega$ with
\begin{equation}\label{eq:separation}
\log\left(-\frac{1}{\rho(w)}\right)
\ge F^{\circ(n+1)}\left(\log\left(-\frac{1}{\rho(z)}\right)\right)
\end{equation}
we have
\[
d_{B,\Omega}(z,w)\ge \arccos\frac{1}{\sqrt{8}}.
\]
Here $F^{\circ m}$ denotes the $m$-fold composition of $F$.
\end{proposition}

\begin{proof}
Fix $\beta=1/2$ in \eqref{eq:split-mass}. Choose $0<\lambda<\min\{1,\varepsilon_\beta\}$ and put $\varepsilon=\lambda\mu(z)$. By the definition of $\mu$ (see \eqref{eq:mu}), we have
\begin{equation}\label{eq:epsilon-choice}
\log\frac{1}{\varepsilon}
=\log\left(\frac{1}{-\rho(z)}\right)+\log\left(1+\vartheta\left(\log\left(\frac{1}{-\rho(z)}\right)+a\right)\right)+\log\frac{1}{\lambda c_0}.
\end{equation}
In what follows, we take $\lambda$ so small that
\[
C_*\lambda^\beta\le\frac{1}{16L_n}\qquad\text{and}\qquad \log\frac1{\lambda c_0}>0.
\]
Thus the first term of \eqref{eq:split-mass} is now at most $(16L_n)^{-1}$.

Choose $M>0$ with $C_{**}e^{-M}\le(16L_n)^{-1}$, and then choose $0<\kappa<1$ such that $c_V\kappa^{-1/n}\ge2n+M$. Take
\[
b\ge\max\left\{1,\log\frac{1}{\lambda c_0},\frac1n\log\frac1\kappa\right\},\qquad H(s)=b+\log(1+\vartheta(s+a)).
\]
Thus $F(s)=s+H(s)$, with $H$ positive and nondecreasing. Induction gives
\begin{equation}\label{eq:iterate-absorption}
F^{\circ m}(s)\ge s+mH(s),\qquad \forall\,m\ge1.
\end{equation}
By \eqref{eq:epsilon-choice}, we have
\[
\log\frac{1}{\varepsilon}
\le F\left(\log\left(-\frac{1}{\rho(z)}\right)\right),
\]
so that
\[
\log\frac{1}{\varepsilon}+nH\left(\log\frac{1}{\varepsilon}\right)
\le F^{\circ n}\left(\log\frac{1}{\varepsilon}\right)
\le F^{\circ (n+1)}\left(\log\left(-\frac{1}{\rho(z)}\right)\right)
\le \log\left(-\frac{1}{\rho(w)}\right).
\]
Thus
\begin{equation}\label{eq:sufficient-rho}
-\rho(w)\le\frac{\kappa\varepsilon}{(1+\vartheta(\log(1/\varepsilon)+a))^n}.
\end{equation}
By \eqref{eq:sufficient-rho}, the exponent in the second term of \eqref{eq:split-mass} is at most
\[
2n\vartheta\left(\log\frac{1}{\varepsilon}+a\right)-(2n+M)\left(1+\vartheta\left(\log\frac{1}{\varepsilon}+a\right)\right)\le-M.
\]
Consequently
\begin{equation}\label{eq:small-mass}
\int_{A_\Omega(w,-1)}|k_z|^2\,\mathrm{d} V\le\frac1{8L_n}.
\end{equation}

The extremal property of the Bergman kernel on $A_\Omega(w,-1)$ gives
\[
|k_z(w)|^2\le K_{A_\Omega(w,-1)}(w)\int_{A_\Omega(w,-1)}|k_z|^2\,\mathrm{d} V.
\]
Using Lemma \ref{lem:localization1} and \eqref{eq:small-mass}, we obtain
\begin{equation}\label{eq:correlation}
\frac{|K_\Omega(z,w)|^2}{K_\Omega(z)K_\Omega(w)}
=\frac{|k_z(w)|^2}{K_\Omega(w)}\le\frac18.
\end{equation}
Finally, the Bergman kernel map into projective Hilbert space pulls back the Fubini--Study metric to \eqref{eq:metric} (cf. Kobayashi \cite{Kobayashi59}). The distance between the lines represented by $k_z$ and $k_w$ is $\arccos|\langle k_z,k_w\rangle|$. Thus
\[
d_{B,\Omega}(z,w)\ge
\arccos\frac{|K_\Omega(z,w)|}{\sqrt{K_\Omega(z)K_\Omega(w)}}
\ge \arccos\frac{1}{\sqrt{8}}.\qedhere
\]
\end{proof}

\begin{proof}[Proof of Theorem~\ref{thm:main}]
Choose $a,b$ as in Proposition~\ref{prop:separation}. Take $\tau_0\geq a+1$ and set $s_0:=\tau_0-a$. Define a sequence $\{s_k\}$ by
\[
s_{k+1}=F(s_k),\qquad k=0,1,2,\cdots.
\]
Then $\tau_k=s_k+a$. Since $a\geq A$, if $\eta(\log(1/\delta_\Omega(z)))\geq\tau_k$, then
\[
\log\left(-\frac{1}{\rho(z)}\right)
\geq\eta\left(\log\frac{1}{\delta_\Omega(z)}\right)-A
\geq s_k.
\]

Let $\gamma:[0,1]\to\Omega$ be any piecewise $C^1$ curve joining $z_0$ to $z$. Write $m:=\left\lfloor{k/(n+1)}\right\rfloor$. Recall that $z_0\in B$, i.e., $-\rho(z_0)=1$. If $-\rho(z)\le e^{-1}$, the continuous function $-\rho\circ\gamma$ starts at one and ends at or below each of the levels $e^{-s_{(n+1)j}}$, $0\le j\le m$. We may therefore choose successive parameters $0<t_0<\cdots<t_m\le1$ with
\[
-\rho(\gamma(t_j))=e^{-s_{(n+1)j}},\qquad 0\le j\le m.
\]
Write $d^*:=\arccos(1/\sqrt{8})$ for simplicity. Since $s_{(n+1)(j+1)}=F^{\circ(n+1)}(s_{(n+1)j})$, the length of each subarc for $t\in[t_j,t_{j+1}]$ is therefore at least $d_*$, in view of Proposition \ref{prop:separation}. Thus
\begin{equation}\label{eq:length_curve}
\ell_{B,\Omega}(\gamma)
\ge md_*
\ge\frac{d_*}{n+1}k-d_*.
\end{equation}
When $|\rho(z)|>e^{-1}$, \eqref{eq:length_curve} holds trivially because $k=0$. Taking the infimum over all such curves gives
\[
d_{B,\Omega}(z_0,z)\ge\frac{d_*}{n+1}k-d_*.\qedhere
\]
\end{proof}

\section{Integral estimates}\label{sec:example}

\begin{proof}[Proof of Corollary~\ref{cor:integral}]
Recall that $H(s)=F(s)-s=b+\log(1+\vartheta(s+a))$. Since $H$ is nondecreasing, we have
\[
\int_{s_l}^{s_{l+1}}\frac{\mathrm ds}{H(s)}
\le\frac{s_{l+1}-s_l}{H(s_l)}=1.
\]
Replace $s$ by $s-a$. Since $\tau_l=s_l+a$, we get
\[
\int_{\tau_l}^{\tau_{l+1}}\frac{\mathrm ds}{b+\log(1+\vartheta(s))}
\le1.
\]
Take $k:=\max\{j:\tau_j\le \eta(\log(1/\delta_\Omega(z)))\}$. It follows that
\begin{align}\label{eq:count-integral}
k
&\ge \int_{\tau_0}^{\eta(\log(1/\delta_\Omega(z)))}\frac{\mathrm ds}{b+\log(1+\vartheta(s))}-1\notag\\
&\gtrsim \int_{\tau_0}^{\eta(\log(1/\delta_\Omega(z)))}\frac{\mathrm ds}{b+\log(1+\vartheta(s))}.
\end{align}
Since $\eta(\log(1/\delta_\Omega(z)))\geq \tau_k$, we infer from \eqref{eq:main} and \eqref{eq:count-integral} that
\[
d_{B,\Omega}(z_0,z)
\gtrsim\int_{\tau_0}^{\eta(\log(1/\delta_\Omega(z)))}\frac{\mathrm ds}{b+\log(1+\vartheta(s))}.
\]
Finally, since $\vartheta(s)\ge t_0\ge1$, we have
\[
b+\log(1+\vartheta(s))\asymp1+\log \vartheta(s).
\]
Here and below $f\asymp g$ means that $f\lesssim g$ and $g\lesssim f$. Taking $s_0=\tau_0$ proves \eqref{eq:integral}.

When $\eta$ is continuous and strictly increasing, $\vartheta(\eta(t))=t$ for $t\ge t_0$. The substitution $v=\eta(t)$ therefore gives \eqref{eq:differential}.
\end{proof}

Table~\ref{tab:moduli} records some consequences of \eqref{eq:differential} and \eqref{eq:eta-loglog} by straightforward calculations:

\begin{table}[ht]
\centering
\caption{Distance lower bounds.}
\label{tab:moduli}
\renewcommand{\arraystretch}{1.2}
\begin{tabular}{@{}lll@{}}
\toprule
$\eta(t)$ & Parameters & Lower bound\\
\midrule
$\alpha t$ & $\alpha>0$ & $\log(1/\delta_\Omega)/\log \log(1/\delta_\Omega)$\\
$t^\theta$ & $0<\theta<1$ & $(\log(1/\delta_\Omega))^\theta/\log\log(1/\delta_\Omega)$\\
$(\log t)^\beta$ & $\beta>1$ & $(\log\log(1/\delta_\Omega))^{\beta-1}$\\
$\alpha\log t$ & $\alpha>0$ & $\log\log\log(1/\delta_\Omega)$\\
\bottomrule
\end{tabular}
\end{table}

\section{Hartogs domains realizing sublinear decay}\label{sec:Hartogs}

We now give examples satisfying \eqref{eq:assumption} for every
nondecreasing function $\eta$ with $\eta(t)\to\infty$ and
$\eta(t)=o(t)$, which proves Theorem \ref{thm:eta-choice}. The following lemma will be used in this construction.

\begin{lemma}\label{lem:hartogs-profile}
Let $\eta:[t_0,\infty)\to(0,\infty)$ be nondecreasing and
$\eta(t)=o(t)$. There is a smooth increasing convex function
$f:[1,\infty)\to(0,\infty)$ such that
\begin{equation}\label{eq:profile-bounds}
f(t)\ge2t,\qquad 2\le f'(t)\le t+2,\qquad \lim_{t\rightarrow+\infty}\frac{f(t)}{t}=+\infty,
\end{equation}
and
\begin{equation}\label{eq:profile-eta}
\eta(3f(t))\le \frac{t}{2}
\end{equation}
for all sufficiently large $t$.
\end{lemma}

\begin{proof}
Fix a smooth convex function $\chi$ on $\mathbb R$ such that $\chi=0$ on $(-\infty,0]$, $0\le\chi'\le1$, and $\chi'=1$ on $[1,\infty)$. Then $\max\{0,x-1\}\le\chi(x)\le x$ for $x\ge0$. Choose an increasing sequence $t_j\rightarrow+\infty$, with $t_j\ge j+1$, so large that
\[
\eta(3(j+2)t)\le t/2,\qquad\forall\,t\ge t_j,\ \ j\ge1.
\]
This is possible for each $j$ because $\eta(t)/t\to0$.
Define
\begin{equation}\label{eq:profile-construction}
f(t)=2t+\sum_{j\ge1}\chi(t-t_j).
\end{equation}
The sum is locally finite, so $f$ is smooth, increasing and convex. Clearly, $f(t)\geq 2t$. Since $t_j\ge j+1$, we have $t<t_j$ whenever $j\geq\lfloor{t}\rfloor$. Thus there are at most $\lfloor{t}\rfloor$ nonzero terms in the derivative sum at a point $t\ge1$, so that $2\le f'(t)\le t+2$. For any fixed integer $m$, when $t$ is sufficiently large,
\[
 f(t)\ge(m+2)t-\sum_{j=1}^m(t_j+1).
\]
Letting $t\rightarrow+\infty$ and then $m\rightarrow+\infty$ proves $f(t)/t\rightarrow+\infty$. Finally, if $t_j\le t<t_{j+1}$, then $f(t)\le(j+2)t$. Monotonicity of $\eta$ gives
\[
\eta(3f(t))\le\eta(3(j+2)t)\le t/2.
\]
Thus \eqref{eq:profile-eta} holds for $t\ge t_1$.
\end{proof}

Theorem \ref{thm:eta-choice} will follow from the following theorem.

\begin{theorem}\label{thm:hartogs}
Suppose that $\eta:[t_0,\infty)\to(0,\infty)$ is nondecreasing with $\eta(t)\rightarrow+\infty$ and $\eta(t)=o(t)$ as $t\rightarrow+\infty$. Choose $f$ as in
Lemma~\ref{lem:hartogs-profile} and set
\begin{equation}\label{eq:hartogs-radius}
F(z):=\exp\left[-f\left(\log\frac{e}{1-|z|^2}\right)\right],\qquad z\in\mathbb{D}.
\end{equation}
Then the complete Hartogs domain
\begin{equation}\label{eq:hartogs-domain}
\Omega_f:=\{(z,w)\in\mathbb{D}\times\mathbb{C}:|w|<F(z)\}
\end{equation}
is bounded and hyperconvex, where $\mathbb{D}=\{z\in\mathbb{C}:|z|<1\}$ is the unit disc. More precisely,
\begin{equation}\label{eq:hartogs-exhaustion}
u(z,w)=\max\left\{|z|^2-1,\frac{|w|^2}{F(z)^2}-1\right\}
\end{equation}
is a continuous negative plurisubharmonic exhaustion satisfying
\begin{equation}\label{eq:hartogs-realization}
-u(z,w)\le C\exp\left[-\eta\left(\log\frac1{\delta_{\Omega_f}(z,w)}\right)\right]
\end{equation}
near $\partial\Omega_f$. Moreover, the hyperconvexity index of $\Omega_f$ is zero.
\end{theorem}

\begin{proof}
The fibers are balls of positive radius, so $\Omega_f$ is a connected open set. It is bounded since $F\le e^{-2}$. The function $\log(e/(1-|z|^2))$ is subharmonic on $\mathbb{D}$, and $f$ is convex and increasing. Thus $f(\log(e/(1-|z|^2)))$ is subharmonic. On $\mathbb{D}\times\mathbb{C}$ the function
\[
\log|w|^2+2f\left(\log\frac{e}{1-|z|^2}\right)
\]
is plurisubharmonic, with value $-\infty$ at $w=0$. Its exponential, extended by zero there, is the continuous plurisubharmonic function $|w|^2/F(z)^2$. Both entries of the maximum in \eqref{eq:hartogs-exhaustion} are therefore plurisubharmonic, and $-1\le u<0$ on $\Omega_f$. It is not difficult to see that $u$ is an exhaustion function, so that $\Omega_f$ is hyperconvex.

We next compare the boundary distance with $-u$. Notice that $\log(e/(1-|z|^2))\geq1$. By \eqref{eq:profile-bounds},
\begin{align}\label{eq:radius-gradient}
|\nabla F(z)|
&=\frac{2|z|}{1-|z|^2}f'\left(\log\frac{e}{1-|z|^2}\right) \exp\left[-f\left(\log\frac{e}{1-|z|^2}\right)\right]\notag\\
&\le 2\left(\log\frac{e}{1-|z|^2}+2\right)\exp\left[-\log\frac{e}{1-|z|^2}-1\right]\notag\\
&\le 6e^{-2} <1.
\end{align}
Also $F(z)\le \exp(-2\log(e/(1-|z|^2)))\le e^{-2}(1-|z|^2)^2$, so $F$ extends continuously by zero to $\partial\mathbb{D}$. The boundary of $\Omega_f$ consists of the points with $|z|<1$, $|w|=F(z)$, together with $(z,0)$ for $|z|=1$. The function $(z,w)\mapsto F(z)-|w|$ is $\sqrt2$-Lipschitz on $\overline{\mathbb{D}}\times\mathbb{C}$ by \eqref{eq:radius-gradient}, and vanishes on $\partial\Omega_f$. Hence
\begin{equation}\label{eq:fiber-gap}
\frac{F(z)-|w|}{\sqrt2}\le\delta_{\Omega_f}(z,w)\le F(z)-|w|.
\end{equation}
Put
\[
\tau(z,w)=\log\frac1{1-|w|/F(z)}\ge0.
\]
Equation~\eqref{eq:fiber-gap} gives
\begin{equation}\label{eq:hartogs-T}
\log\frac1{\delta_{\Omega_f}(z,w)}\le f\left(\log\frac{e}{1-|z|^2}\right)+\tau(z,w)+\log\sqrt2.
\end{equation}
On the other hand, the definition of $u$ gives the two bounds
\begin{align}
&-u(z,w)\le1-|z|^2=\exp\left[1-\log\frac{e}{1-|z|^2}\right],\label{eq:hartogs-u-bounds1}\\
&-u(z,w)\le1-\frac{|w|^2}{F(z)^2}\le2e^{-\tau(z,w)}.\label{eq:hartogs-u-bounds2}
\end{align}
If necessary, extend $\eta$ constantly to $[0,t_0]$. This does not change the boundary hypothesis. By \eqref{eq:profile-eta} and $\eta(t)=o(t)$, after enlarging a fixed constant $C_0$ we have
\begin{align}
&\eta\left(3f\left(\log\frac{e}{1-|z|^2}\right)\right)\le \frac{1}{2}\log\frac{e}{1-|z|^2}+C_0,\qquad \text{if}\ \log\frac{e}{1-|z|^2}\ge1,\label{eq:hartogs-eta-bounds1}\\
&\eta(3\tau(z,w))\le\frac{\tau(z,w)}{2}+C_0,\qquad \text{if}\ \tau(z,w)\ge0.\label{eq:hartogs-eta-bounds2}
\end{align}
Notice that $f(\log(e/(1-|z|^2)))\ge f(1)\ge 2$. If $\tau(z,w)\le f(\log(e/(1-|z|^2)))$, then \eqref{eq:hartogs-T} gives
\[
\log\frac1{\delta_{\Omega_f}(z,w)}\le 3f\left(\log\frac{e}{1-|z|^2}\right).
\]
Using \eqref{eq:hartogs-u-bounds1} and \eqref{eq:hartogs-eta-bounds1}, we obtain
\[
(-u(z,w))\exp\left[\eta\left(\log\frac1{\delta_{\Omega_f}(z,w)}\right)\right]
\le e^{1+C_0}.
\]
If $\tau(z,w)>f(\log(e/(1-|z|^2)))$, then $\tau>2$ and by \eqref{eq:hartogs-T} we have
\[
\log\frac1{\delta_{\Omega_f}(z,w)}\le 3\tau(z,w).
\]
The bound in \eqref{eq:hartogs-u-bounds2} and \eqref{eq:hartogs-eta-bounds2} gives instead
\[
(-u(z,w))\exp\left[\eta\left(\log\frac1{\delta_{\Omega_f}(z,w)}\right)\right]
\le 2e^{C_0}.
\]
This proves \eqref{eq:hartogs-realization} uniformly near the entire boundary.

Finally, suppose that some negative plurisubharmonic exhaustion $v$ on $\Omega_f$ satisfies $-v\le C\delta_{\Omega_f}^\gamma$ for a fixed $\gamma>0$ near the boundary. Its restriction $z\mapsto v(z,0)$ is a negative subharmonic function on $\mathbb{D}$. Fix $r_0\in(0,1)$ and choose $m>0$ with $v(z,0)\le-m$ on $|z|=r_0$. Comparing $v(z,0)$ with the radial harmonic function $-m\log|z|/\log r_0$ on a fixed annulus $r_0<|z|<1$ gives
\[
-v(r,0)\ge m\frac{\log{r}}{\log r_0}=\frac{m}{-\log r_0}\left(-\log{r}\right)\geq c(1-r^2),\qquad r\rightarrow1^-
\]
for some constant $c>0$ independent of $r$. Taking $z=r$ and $w=0$ in \eqref{eq:fiber-gap}, we get $\delta_{\Omega_f}(r,0)\le F(r)$.
Consequently
\[
c(1-r^2)
\le-v(r,0)
\le CF(r)^\gamma
\le C\exp\left[-\gamma f\left(\log\frac{e}{1-r^2}\right)\right],
\]
which would imply
\[
\gamma f\left(\log\frac{e}{1-r^2}\right)\le \log\frac{e}{1-r^2}+O(1).
\]
This contradicts $f(t)/t\to+\infty$ and proves that the hyperconvexity index is zero.
\end{proof}

\section{An upper bound on a family of complete Hartogs domains}\label{sec:Hartogs-alpha}

We now give an upper bound for the Bergman distance along the axis of a family of complete Hartogs domains. Recall that $\alpha>1$ and
\[
\Omega_\alpha=\left\{(z,w)\in\mathbb C^2:|z|<1,\quad |w|<\exp\left[-\left(\log\frac{e}{1-|z|^2}\right)^\alpha\right]\right\}.
\]
The domain $\Omega_\alpha$ is bounded and hyperconvex. More precisely, we have the following proposition.

\begin{proposition}\label{prop:hartogs-alpha-lower}
$\Omega_\alpha$ admits a continuous negative plurisubharmonic exhaustion function $u_\alpha$ such that
\begin{equation}\label{eq:hartogs-alpha-decay}
-u_\alpha(z,w)
\le3e\exp\left[-\left(\log\frac1{\delta_{\Omega_\alpha}(z,w)}\right)^{1/\alpha}\right].
\end{equation}
\end{proposition}

Consequently, for any fixed $P_0\in\Omega_\alpha$,
\begin{equation}\label{eq:hartogs-alpha-lower}
d_{B,\Omega_\alpha}(P_0,P)
\gtrsim
\frac{(\log(1/\delta_{\Omega_\alpha}(P)))^{1/\alpha}}
{\log\log(1/\delta_{\Omega_\alpha}(P))}
\qquad\text{as }P\rightarrow\partial\Omega_\alpha.
\end{equation}
(See Table \ref{tab:moduli}.)

\begin{proof}
Put
\[
h_\alpha(r):=\begin{cases}
\displaystyle\exp\left[1-\left(\log\frac1r\right)^{1/\alpha}\right], &0<r\le e^{-1},\\
0,&r=0,
\end{cases}
\]
and define
\begin{equation}\label{eq:hartogs-alpha-exhaustion}
u_\alpha(z,w)=|z|^2+h_\alpha(|w|)-1.
\end{equation}
The function
\[
x\longmapsto\exp[1-(-x)^{1/\alpha}],\qquad x\le-1,
\]
is convex and increasing. Thus $h_\alpha(|w|)$ is subharmonic for $0<|w|<e^{-1}$, and its continuous extension across $w=0$ is also subharmonic. It follows that $u_\alpha$ is continuous and plurisubharmonic on $\Omega_\alpha$.

Write
\[
F_\alpha(z)
=\exp\left[-\left(\log\frac e{1-|z|^2}\right)^\alpha\right].
\]
Since $h_\alpha(F_\alpha(z))=1-|z|^2$ and $h_\alpha$ is strictly increasing, we have $-1\le u_\alpha<0$ on $\Omega_\alpha$. Also $F_\alpha$ extends continuously by zero to $\partial\mathbb D$. Hence $u_\alpha$ extends continuously by zero to the entire boundary of $\Omega_\alpha$, including the points $(\zeta,0)$ with $|\zeta|=1$. Therefore $u_\alpha$ is an exhaustion function.

For $0<r\le e^{-1}$, direct differentiation gives
\[
\frac{\mathrm d}{\mathrm dr}\log\frac{h_\alpha(r)}r
=\frac1r\left[\frac1\alpha\left(\log\frac1r\right)^{1/\alpha-1}-1\right]
\le0.
\]
Thus $h_\alpha(r)/r$ is nonincreasing. Since $h_\alpha(e^{-1})=1$, it follows that $h_\alpha(r)\ge r$. For $a,b>0$ with $a+b\le e^{-1}$, the same monotonicity gives
\[
h_\alpha(a)+h_\alpha(b)
\ge\frac a{a+b}h_\alpha(a+b)+\frac b{a+b}h_\alpha(a+b)
=h_\alpha(a+b).
\]
Consequently, by continuity also at zero,
\begin{equation}\label{eq:hartogs-alpha-modulus}
|h_\alpha(a)-h_\alpha(b)|
\le h_\alpha(|a-b|),\qquad 0\le a,b\le e^{-1}.
\end{equation}

Fix $P=(z,w)\in\Omega_\alpha$. It is easy to see that
\[
0<\delta_{\Omega_\alpha}(P)\le F_\alpha(z)-|w|\le e^{-1}.
\]
Choose $Q=(\zeta,\omega)\in\partial\Omega_\alpha$ such that $|P-Q|=\delta_{\Omega_\alpha}(P)$. Since $u_\alpha(Q)=0$, \eqref{eq:hartogs-alpha-modulus} yields
\begin{align*}
-u_\alpha(P)
&=u_\alpha(Q)-u_\alpha(P)\\
&\le\bigl||\zeta|^2-|z|^2\bigr|+|h_\alpha(|\omega|)-h_\alpha(|w|)|\\
&\le2|\zeta-z|+h_\alpha\bigl(\bigl||\omega|-|w|\bigr|\bigr)\\
&\le2\delta_{\Omega_\alpha}(P)+h_\alpha(\delta_{\Omega_\alpha}(P))\\
&\le3h_\alpha(\delta_{\Omega_\alpha}(P)).
\end{align*}
This proves \eqref{eq:hartogs-alpha-decay} uniformly on $\Omega_\alpha$.
\end{proof}

We continue to prove Theorem \ref{thm:hartogs-alpha}. The estimate below concerns points on the axis $w=0$. All constants in this section may depend on $\alpha$. First, we will compare the boundary distance on the axis with the radius of the corresponding fiber.

\begin{lemma}\label{lem:hartogs-alpha-boundary}
For every $z\in\mathbb D$,
\begin{equation}\label{eq:hartogs-alpha-boundary}
\delta_{\Omega_\alpha}(z,0)\asymp \exp\left[-\left(\log\frac{e}{1-|z|^2}\right)^\alpha\right].
\end{equation}
In particular, as $|z|\to1^-$,
\begin{equation}\label{eq:hartogs-alpha-log-distance}
\log\frac1{\delta_{\Omega_\alpha}(z,0)}=\left(\log\frac{e}{1-|z|^2}\right)^\alpha+O(1).
\end{equation}
\end{lemma}

\begin{proof}
For simplicity, write
\[
R_\alpha(s)=\exp\left[-\left(\log\frac{e}{s}\right)^\alpha\right],\quad 0<s\le1,
\]
and
\[
F_\alpha(z):=R_\alpha(1-|z|^2)=\exp\left[-\left(\log\frac{e}{1-|z|^2}\right)^\alpha\right].
\]
The point $(z,F_\alpha(z))$ lies on $\partial\Omega_\alpha$, so $\delta_{\Omega_\alpha}(z,0)\le F_\alpha(z)$. On the other hand,
\[
R_\alpha'(s)=\frac\alpha s\left(\log\frac{e}{s}\right)^{\alpha-1}R_\alpha(s).
\]
Notice that the right-hand side tends to zero as $s\to0^+$. It is also bounded as $s\to1^-$. Consequently $F_\alpha$ has bounded gradient on $\mathbb D$ and extends by zero to a Lipschitz function on $\overline{\mathbb D}$. Let $M_\alpha$ be a Lipschitz constant of $F_\alpha$. The function $G(z,w)=F_\alpha(z)-|w|$ is $\sqrt{1+M_\alpha^2}$-Lipschitz on $\overline{\mathbb D}\times\mathbb C$ and vanishes on $\partial\Omega_\alpha$, including the points $(\zeta,0)$ with $|\zeta|=1$. Hence
\[
F_\alpha(z)
=G(z,0)
\le\sqrt{1+M_\alpha^2}\,\delta_{\Omega_\alpha}(z,0).
\]
This proves both assertions.
\end{proof}

The next lemma supplies the estimate for the metric in the axial direction. Its proof uses the extremal formula for the weighted Bergman metric and a local construction of a holomorphic test function.

\begin{lemma}\label{lem:hartogs-alpha-metric}
As $|z|\to1^-$,
\begin{equation}\label{eq:hartogs-alpha-metric}
\beta_{\Omega_\alpha}((z,0);(1,0))
\lesssim\frac{1}{1-|z|^2}\left(\log\frac{e}{1-|z|^2}\right)^{(\alpha-1)/2}.
\end{equation}
\end{lemma}

\begin{proof}
Put $\varphi(z)=2(\log(e/(1-|z|^2)))^\alpha$, and let $K_\varphi$ denote the diagonal kernel of
\[
A^2_\varphi(\mathbb D)=\left\{f\in\mathcal O(\mathbb D):\|f\|_\varphi^2:=\int_{\mathbb D}|f|^2e^{-\varphi}<\infty\right\}.
\]
By Ligocka's formula (cf. \cite{L89}), we have
\[
K_{\Omega_\alpha}((z,0))=\pi^{-1}K_\varphi(z),
\]
so that
\begin{equation}\label{eq:hartogs-alpha-weighted-kernel}
\beta_{\Omega_\alpha}((z,0);(1,0))^2 =\partial_z\partial_{\bar z}\log K_\varphi(z).
\end{equation}
By the extremal formula for the Bergman metric, we have
\begin{equation}\label{eq:hartogs-alpha-metric-extremal}
\partial_z\partial_{\bar z}\log K_\varphi(z)=\frac{J_\varphi(z)}{K_\varphi(z)},\qquad J_\varphi(z)=\sup\{|f'(z)|^2:f(z)=0,\ \|f\|_\varphi\le1\}.
\end{equation}
We shall prove an upper bound for $J_\varphi$ and a lower bound for
$K_\varphi$ as $|z|\rightarrow1-$.

Direct differentiation gives
\begin{equation}\label{eq:hartogs-alpha-laplacian}
\varphi_{z\bar z}(z)
=\frac{2\alpha}{(1-|z|^2)^2}\left(\left(\log\frac{e}{1-|z|^2}\right)^{\alpha-1}+(\alpha-1)|z|^2\left(\log\frac{e}{1-|z|^2}\right)^{\alpha-2}\right)
\end{equation}
and
\begin{equation}\label{eq:hartogs-alpha-hessian}
\varphi_{zz}(z)
=\frac{2\alpha\bar{z}^2}{(1-|z|^2)^2}\left(\left(\log\frac{e}{1-|z|^2}\right)^{\alpha-1}+(\alpha-1)\left(\log\frac{e}{1-|z|^2}\right)^{\alpha-2}\right).
\end{equation}
Put $s(z):=(1-|z|^2)(\log(e/(1-|z|^2)))^{-(\alpha-1)/2}$. Then \eqref{eq:hartogs-alpha-laplacian} and \eqref{eq:hartogs-alpha-hessian} imply
\begin{equation}\label{eq:hartogs-alpha-local-curvature-z}
\varphi_{z\bar z}(z)\asymp s(z)^{-2},\qquad |D^2\varphi(z)|\asymp s(z)^{-2}.
\end{equation}
Here $D^2\varphi$ is the real Hessian. Since $s(z)=o(1-|z|)$ as $|z|\rightarrow1-$, for any $A>2$, we may take $z$ sufficiently close to the boundary so that $D(z,2As(z))\subset\mathbb D$. Moreover, for all $\zeta\in D(z,2As(z))$, we have $s(\zeta)\asymp s(z)$. This combined with \eqref{eq:hartogs-alpha-local-curvature-z} yields
\begin{equation}\label{eq:hartogs-alpha-local-curvature}
\varphi_{\zeta\bar{\zeta}}(\zeta)\asymp s(z)^{-2},\qquad |D^2\varphi(\zeta)|\lesssim s(z)^{-2}
\end{equation}
as $|z|\rightarrow1-$. Here the implicit constants can be chosen independent of $A$. Let
\[
H_z(\zeta)=\frac{\varphi(z)}2+\varphi_z(z)(\zeta-z).
\]
Taylor's formula and \eqref{eq:hartogs-alpha-local-curvature} imply
\begin{equation}\label{eq:hartogs-alpha-affine-weight}
|\varphi(\zeta)-2\operatorname{Re}H_z(\zeta)|\le C, \qquad \zeta\in D(z,s(z)).
\end{equation}
If $f(z)=0$ and $\|f\|_\varphi\le1$, set $h=f e^{-H_z}$. Then $h(z)=0$ and
\[
\int_{D(z,s(z))}|h|^2 \le e^C\int_{D(z,s(z))}|f|^2e^{-\varphi}\le e^C.
\]
The local $L^2$ Cauchy estimate gives $|h'(z)|^2\lesssim s(z)^{-4}$. Since
$h'(z)=f'(z)e^{-\varphi(z)/2}$, we obtain $|f'(z)|^2\lesssim s(z)^{-4}e^{\varphi(z)}$, and hence
\begin{equation}\label{eq:hartogs-alpha-derivative-upper}
J_\varphi(z)\lesssim e^{\varphi(z)}s(z)^{-4}.
\end{equation}

Next, we shall construct $F_z\in A^2_\varphi(\mathbb D)$ with
\begin{equation}\label{eq:hartogs-alpha-test-function}
F_z(z)=1,\qquad \|F_z\|_\varphi^2\lesssim s(z)^2e^{-\varphi(z)},
\end{equation}
so that $K_\varphi(z)\gtrsim s(z)^{-2}e^{\varphi(z)}$. Consider the holomorphic function
\[
g_z(\zeta)=\exp[H_z(\zeta)-H_z(z)].
\]
It follows that $g_z(z)=1$, and \eqref{eq:hartogs-alpha-affine-weight} gives
\begin{equation}\label{eq:hartogs-alpha-local-test}
|g_z(\zeta)|^2e^{-\varphi(\zeta)}\lesssim e^{-\varphi(z)}, \qquad \zeta\in D(z,s(z)).
\end{equation}
Choose $\chi_z\in C_0^\infty(D(z,s(z)))$ with $0\le\chi_z\le1$, $\chi_z=1$ on $D(z,s(z)/2)$, and $|\bar\partial\chi_z|\lesssim s(z)^{-1}$. Then $\chi_zg_z\in C^\infty_0(\mathbb{D})$. Moreover, fix a smooth function $\theta:[0,\infty)\to[0,1]$ with $\theta\equiv1$ on $[0,1]$ and $\theta\equiv0$ on $[2,\infty)$, and put
\[
b_z:=\theta\left(\frac{|\zeta-z|}{As(z)}\right)\log\frac{|\zeta-z|^2}{A^2s(z)^2},\qquad \Phi_z(\zeta):=\varphi(\zeta)+b_z,
\]
where $A>2$. Notice that the function $b_z$ is subharmonic on $D(z,As(z))$, harmonic on $D(z,As(z))\setminus\{z\}$ and vanishes outside $D(z,2As(z))$. On $D(z,2As(z))\setminus D(z,As(z))$, straightforward calculation gives
\[
|\partial_\zeta\partial_{\bar\zeta}b_z|\lesssim \frac{1}{A^2s(z)^2}.
\]
Choose $A>2$ so large that $\Phi_z$ is subharmonic on $\mathbb D$ with a logarithmic pole at $z$, in view of \eqref{eq:hartogs-alpha-local-curvature}. On the support of
$\bar\partial\chi_z$, we have
\begin{equation}\label{eq:hartogs-alpha-supp-chi}
(\Phi_z)_{\zeta\bar\zeta}=\varphi_{\zeta\bar\zeta}\gtrsim c_0s(z)^{-2},\qquad e^{-b_z}\le4A^2.
\end{equation}
The weighted H\"ormander $L^2$ estimate (cf. \cite{H65}) therefore supplies a function $v_z$ with
\begin{equation}\label{eq:hartogs-alpha-dbar}
\bar{\partial}v_z=g_z\bar\partial\chi_z,
\end{equation}
satisfying
\begin{equation}\label{eq:hartogs-alpha-correction}
\int_{\mathbb D}|v_z|^2e^{-\Phi_z}
\lesssim \int_{\mathbb D}\frac{|g_z\bar\partial\chi_z|^2}{(\Phi_z)_{\zeta\bar\zeta}}e^{-\Phi_z}
\lesssim s(z)^2e^{-\varphi(z)}<+\infty.
\end{equation}
The last inequality uses \eqref{eq:hartogs-alpha-local-test}, \eqref{eq:hartogs-alpha-supp-chi} and the fact that the support of $\bar\partial\chi_z$ has area at most $Cs(z)^2$.
The finite integral in \eqref{eq:hartogs-alpha-correction} forces $v_z(z)=0$.
Moreover, since $b_z\le\log4$ on $\mathbb D$, we have
\begin{equation}\label{eq:hartogs-alpha-weightednorm-v}
\int_{\mathbb D}|v_z|^2e^{-\varphi}
\le4\int_{\mathbb D}|v_z|^2e^{-\Phi_z}
\lesssim s(z)^2e^{-\varphi(z)}.
\end{equation}
Consequently $F_z=\chi_z g_z-v_z$ is holomorphic on $\mathbb D$,
$F_z(z)=1$, and
\[
\|F_z\|_\varphi^2
\lesssim \|\chi_zg_z\|_\varphi^2+\|v_z\|_\varphi^2
\lesssim s(z)^2e^{-\varphi(z)}
\]
by \eqref{eq:hartogs-alpha-local-test}  and \eqref{eq:hartogs-alpha-weightednorm-v}, i.e., \eqref{eq:hartogs-alpha-test-function} holds. The extremal characterization of the weighted Bergman kernel now yields
\begin{equation}\label{eq:hartogs-alpha-diagonal}
K_\varphi(z)\ge\frac{|F_z(z)|^2}{\|F_z\|_\varphi^2} \gtrsim s(z)^{-2}e^{\varphi(z)}.
\end{equation}

Substituting \eqref{eq:hartogs-alpha-derivative-upper} and \eqref{eq:hartogs-alpha-diagonal} into \eqref{eq:hartogs-alpha-metric-extremal}, we obtain
\[
\partial_z\partial_{\bar z}\log K_\varphi(z)
\lesssim\frac{s(z)^{-4}e^{\varphi(z)}}{s(z)^{-2}e^{\varphi(z)}}
\lesssim s(z)^{-2}
\asymp \frac{1}{(1-|z|^2)^2}\left(\log\frac{e}{1-|z|^2}\right)^{\alpha-1}.
\]
Taking square roots in \eqref{eq:hartogs-alpha-weighted-kernel} proves \eqref{eq:hartogs-alpha-metric}.
\end{proof}

\begin{proof}[Proof of Theorem~\ref{thm:hartogs-alpha}]
Consider the radial curve $\gamma(r)=(r,0)$, $0\le r\le1-t$. The Bergman metric is bounded on each compact subset of $\Omega_\alpha$. Lemma~\ref{lem:hartogs-alpha-metric} therefore gives
\begin{align*}
d_{B,\Omega_\alpha}((0,0),P_t)
&\le\ell_{B,\Omega_\alpha}(\gamma)\\
&\lesssim1+\int_{1/2}^{1-t}\frac{(\log(e/(1-r^2)))^{(\alpha-1)/2}}{1-r^2}\,\mathrm dr.
\end{align*}
Make the substitution $s=\log(e/(1-r^2))$, so that $\mathrm ds =2r\,\mathrm dr/(1-r^2)$. Then the last expression is bounded (up to multiplying a constant) by
\[
1+\int_{\log(4e/3)}^{\log(e/(2t-t^2))}s^{(\alpha-1)/2}\,\mathrm ds
\lesssim 1+\left(\log\frac{1}{2t-t^2}\right)^{(\alpha+1)/2}
\lesssim \left(\log\frac{1}{t}\right)^{(\alpha+1)/2}
\]
as $t\to0^+$, which proves
\[
d_{B,\Omega_\alpha}((0,0),P_t)\lesssim\left(\log\frac1t\right)^{(\alpha+1)/2}.
\]
This combined with \eqref{eq:hartogs-alpha-log-distance} and \eqref{eq:hartogs-alpha-lower} gives \eqref{eq:hartogs-alpha-distance-delta}.
\end{proof}

{\bf Acknowledgements.}\quad The authors are grateful to ChatGPT 6 Astra for suggesting the use of the volume-capacity inequality in \cite{ACKPZ09} instead of the ``quasi-symmetry'' of the pluricomplex Green function, constructing the example in \S\,\ref{sec:Hartogs} and  language editing. They reviewed and verified the material they adopted and take final responsibility for all arguments and conclusions in the paper and for its adherence to standards of academic integrity.

\end{document}